\documentclass[12pt,a4paper]{article}

\usepackage{amsmath}
\usepackage{amssymb}
\usepackage{amsthm}
\usepackage{graphicx}
\usepackage{ucs}
\usepackage[utf8x]{inputenc}

\DeclareMathOperator{\graph}{graph}

\DeclareMathOperator{\ind}{ind}

\newcommand{\Rn}{{\mathbb R}^n}
\newcommand{\Rm}{{\mathbb R}^m}

\newcommand{\Z}{{\mathbb Z}}
\newcommand{\R}{{\mathbb R}}

\newcommand{\cF}{{\mathcal F}}

\newcommand{\cM}{{\mathcal M}}

\newcommand{\cP}{{\mathcal P}}

\newcommand{\cS}{{\mathcal S}}

\newcommand{\cV}{{\mathcal V}}

\newcommand{\ep}{{\epsilon}}

\newtheorem{theorem}{Theorem}
\newtheorem{prop}{Proposition}
\newtheorem*{claim}{Claim}
\newtheorem{cor}{Corollary}
\newtheorem{lemma}{Lemma}

\theoremstyle{definition}
\newtheorem{definition}{Definition}[section]
\newtheorem{example}{Example}

\theoremstyle{remark}

\newtheorem*{remark}{Remark}

\begin{document}

\author{EFTHIMIOS KAPPOS\\ 
School of General Sciences,\footnote{Submitted for publication, 2009, under revision} \\ Department of Mathematics, \\
Faculty of Engineering, \\ Aristotle University of Thessaloniki, Greece.}

\title{\bf \huge Topological Necessary Conditions for Control Dynamics}

\date{}

\maketitle

\section{Introduction}

Suppose a smooth feedback control has been found so that the controlled dynamics have an asymptotically stable attractor at some point x in state space. Then local Lyapunov functions exist for the dynamics; these functions must all have a unique minimum at the point x, but are otherwise arbitrary. On a compact level set of any such Lyapunov function, the controlled dynamics point inwards, in other words in the direction of the negative of the gradient of the Lyapunov function. As maps to the unit sphere, the two vector fields thus have the same degree. But the degree of the negative gradient vector field is known —it is exactly $(−1)^n \neq 0$. Hence the map of the controlled dynamics, restricted to a level set, to the sphere must be onto. This means roughly that all control directions must be available near a point that is to be stabilized by control. In control theory, this is known as the Brockett condition, but such simple de- gree results were widely known before (Krasnosel’skii’s name is mentioned in conjunction with that of Brockett.) 
Such degree-theoretic arguments have been used for some time in topology and were eventually adopted by control theorists to derive a number of related necessary conditions for design of controlled dynamics using continuous feed- back. 

In this paper, we give an account of this theory that has two distinguishing features: first, there is really no reason to limit ourselves to local results: a global theory is straightforward to obtain. Secondly, we point out that such necessary conditions are in some funda- mental sense of limited value; this is because they involve maps from manifolds to spheres of the same dimension. By the Hopf theory, homotopy equivalence classes of such maps are completely classified by a single integer (traditionally called the degree, but better interpreted in terms of homology groups.) 

Several comments are in order: generalizations of the Brockett 
necessary condition have been obtained by Coron and others. These 
suffer to some extent by the problem mentioned above, but have helped 
clarify the fact that \emph{surjectivity} is not enough: the vector field 
cannot `twist' too much either (examples use the degree $k$ maps $\dot 
z = z^{k}$ as part of a control system decomposition.)
Moreover, extensions to the case of \emph{dynamic feedback} have been derived.

Perhaps more importantly, many systems cannot be stabilized using 
continuous feedback, but can be easily stabilized with {\bf discontinuous} 
feedback.
The discontinuity is rather mild: it is usually limited to a `thin' 
subset of state space.
Recent work of Sontag, Clarke, Subbotin and others has led to a theory 
of discontinuous feedback controls and a methodology for obtaining 
nonsmooth Lyapunov functions.
Now it is possible to interpret this theory in a {\emph{hamiltonian} 
context: The discontinuities correspond to jumps between locally 
nonsingularly projected lagrangian levels.
This way of examining possibly discontinuous feedback controls is 
conceptually easier to understand and is in step with the philosophy 
of the book~\cite{contbk} which is to give, as far as is possible, \emph{geometric} 
accounts of \emph{analytical} points.

We begin by giving an outline of the algebraic 
topological machinery needed for a discussion of necessary conditions.
Here, we depart from the practice of delegating mathematical background to an appendix, because we believe that this theory is quite accessible and elegant.

A collection of \emph{global necessary conditions} is then given, 
directly based on the topological results.
Essentially, it is argued that if certain dynamics are achieved, 
then {\bf index-theoretic} conditions can be deduced by counting the 
equilibrium points and their stability (Euler-type arguments) and 
{\emph{degree-theoretic} results are obtained by the Hopf theorems 
using the Gauss maps of the dynamics and the gradient vector field of 
Lyapunov functions. 

Finally, let us point out that this paper by no means exhausts the 
theory of necessary conditions in control design.
Much more crucial limitations on achievable control dynamics arise 
from the theory of {\bf feedback invariant} objects, a theory that 
is developed in the book~\cite{contbk}.

\section{Some Background and Methods from Algebraic Topology}

A thumb-nail sketch of a number of concepts and methods from algebraic 
topology will now be given.
There is no effort to be rigorous, but we do hope to explain enough 
about the computational methods so that a non-expert reader can use 
them in concrete situations.

\emph{Algebraic topology} is based on a simple principle: attach algebraic 
objects to topological spaces that are invariants of the homotopy 
type of the space.
Thus, more precisely we assign algebraic objects to 
homotopy equivalence classes of spaces and this assignment is 
`functorial' in the sense that maps of spaces induce homomorphisms of 
the algebraic objects.
This already gives useful tests: since homotopy equivalent spaces 
have isomorphic algebraic objects, two spaces are definitely not 
homotopy equivalent if their algebraic objects are not isomorphic.
The bulk of algebraic topology consists of deriving finer and finer 
such objects so as to be able to better distinguish spaces and in 
making clever use of its basic constructions to aid the analysis of 
global aspects of other subjects (such as complex analysis, pdes, 
geometry etc.) 

\paragraph{Singular homology}

The easiest algebraic object we can attach to a space is the graded 
abelian group $H_{*} (X)$ called the {\bf singular homology group} 
of $X$ (with integer coefficients.)
One can get quite far with only a vague understanding of what the 
singular homology measures and the reason is that powerful and 
effective methods for the computation of $H_{*} (X)$ exist.
We outline the Mayer-Vietoris sequence and explain the concept of a 
long exact sequence of a pair and its relation to excision.

A graded abelian group $G = \oplus G_{k}$ is a direct sum of groups 
$G_{k}$, $k \in \Z _{+}$, such that the group addition is 
`component-wise', in other words we add elements belonging to the same 
graded component together.
The notation
\[ g = \ldots + g^{0} + g^{1} + \ldots \]
for an element of $G$, where $g^{k} \in G_{k}$ is therefore unambiguous.

For a topological space, the $k$th homology group $H_{k} (X; \Z )$ 
measures in some sense the {\emph{`holes'} of $X$ that are like $k$ 
spheres $S^{k}$ (think of a boundary-less space, like a sphere, that 
does not actually bound anything itself in $X$.)
The $0$th group $H_{0} (X)$ is equal to $\Z$ if $X$ is path connected.
(Recall that a $0$-sphere is the boundary of an interval, i.e. the 
union of two points.)
There is a way of defining \emph{reduced homology groups} $\tilde 
H_{k} (X)$ so that $\tilde H_{0} (X) = 0$ for a connected space and 
so that all higher dimensional groups coincide with the non-reduced 
ones.

Let us give some examples (we omit the zeroth homology group.)
The singular homology of the circle $S^{1}$ is $H_{1} ( S^{1} ) \simeq \Z$ 
and zero for $k > 1$.
Since $\pi _{1} ( S^{1} ) = \Z$ also, the homology group contains the 
same information as the fundamental group of the circle.
Note the difference in interpretation, though:
In the former case (for $\pi _{1}$), we are thinking of maps from the 
circle to itself, 
classifed by the number of net encirclements
In the latter, we are 
thinking of a fixed circle ---coinciding in this case with the whole 
space $S^{1}$--- as the generator of a free abelian group; in this 
sense, we can write
\[ H_{1} ( S^{1} ) = \Z [ S^{1} ] \simeq \Z . \]

For the sphere $S^{m}$ of dimension $m >1$, $H_{m} ( S^{m} ) \simeq 
\Z$ is the only nonzero homology group in positive dimension.
Since we also have that the $m$th homotopy group of an $m$ sphere 
is $\Z$, we have not yet obtained anything new, compared with homotopy 
theory.
This is a little misleading: homotopy is both subtler than homology 
and far more difficult to compute: we do not, even today, have a 
complete list of the homotopy groups of spheres.
Moreover, $\pi _{m+k} ( S^{m} ) $ may very well be nonzero for some 
$k >0$, while $H_{m+k} ( S^{m} ) = 0$ always.

The `coincidence' is really due to a nontrivial theorem, the {\bf Hurewicz 
isomorphism} that states that homotopy and homology groups are 
isomorphic at the first level when one, and hence both, are 
nontrivial (the abelianization of the possibly nonabelian fundamental 
group is to be considered, if this happens at the first level.)

A quick check that homology theory does indeed give something new is 
to compute the homology of the torus $T^{2}$.
We have that $H_{2} ( T^{2} ) \simeq \Z$ even though $\pi _{2} ( 
T^{2} ) =0$! (it may be profitable to spend a minute or two pondering 
the difference.)

\paragraph{The Mayer-Vietoris sequence}

Supose a space $X$ is the union of two opne subsets, $X= A \cup B$, 
with $A,B$ opne and $A \cap B \ne \emptyset$.
Then there is a \emph{long exact sequence} involving the homology groups of 
the three spaces
\begin{equation}
\ldots \to H_{k} ( A \cap B ) \to H_{k} (A) \oplus H_{k} (B) \to 
H_{k} (X) \to H_{k-1} ( A \cap B ) \to \ldots           
        \label{leshom}
\end{equation}
This gives a surprisingly powerful tool for the computation of homology.
Even without knowing what the maps at each stage are (for which we 
refer the reader to standard accounts such as~\cite{gr:harp}) the 
exactness allows the computation in concrete cases, such as that of 
the spheres.
For this, decompose an $m$-sphere into two slightly overlapping 
hemispheres $A$ and $B$ so that their intersection $A \cap B$ is deformable to a 
sphere of dimension $m-1$.
We can start an induction with dimension $m=1$ and use the 
Mayer-Vietoris sequence to obtain
\begin{equation}
\ldots \to H_{k} ( S^{m-1} ) \to H_{k} (A) \oplus H_{k} (B) \to 
H_{k} ( S^{m} ) \to H_{k-1} ( S^{m-1} ) \to \ldots              
        \label{lesSm}
\end{equation}
yielding, for $k=m$, and since disks have no homology
\begin{equation}
\ldots \to 0 \to 0 \oplus 0 \to 
H_{m} ( S^{m} ) \to \Z \to 0 \to \ldots .
\end{equation}
We conclude that $H_{m} ( S^{m} ) \simeq \Z$, since any exact sequence 
of the form
\[ 0 \to C \to D \to 0 \]
implies that the middle map is one-to-one and onto, i.e. an isomorphism.

\paragraph{Long exact sequence of a pair and excision}

A second very useful method for the computation of homology comes from 
considering pairs $(X,A)$, where $A \subset X$ is a subspace.
One gets the long exact sequence for the pair
\begin{equation}
\ldots \to H_{k} ( A ) \to H_{k} (X) \to 
H_{k} ( X,A ) \to H_{k-1} ( A ) \to \ldots              
        \label{lesPair}
\end{equation}
where the groups $H_{k} (X,A)$ are the {\bf relative homology groups}.
Without giving the exact definition, found in the standard texts, let 
us mention that in many important cases, these relative groups are 
isomorphic to the homology groups of the quotient space $X/A$ (see 
chapter 3 for the definition.)
The long exact sequence for a pair is thus extremely useful for the 
computation of the homological Conley index.

As an example, let us show that the quotient $D^{n} / S^{n-1}$ of a 
closed ball by its bounding sphere has the homology of the 
$n$-sphere $S^{n}$.
The long exact sequence of the pair $( D^{n} , S^{n-1} )$ is
\begin{equation}
\ldots \to H_{k} ( D^{n} ) \to 
H_{k} ( D^{n} , S^{n-1} ) \to H_{k-1} ( S^{n-1} ) \to H_{k-1} ( 
D^{n} ) \to \ldots
\end{equation}
and so, at $k=n$, we get
\begin{equation}
\ldots \to 0 \to 
H_{k} ( D^{n} , S^{n-1} ) \to \Z  \to 0 \to \ldots
\end{equation}
hence $ H_{k} ( D^{n} , S^{n-1} ) \simeq H_{k} ( D^{n} / S^{n-1} ) 
\simeq \Z$.
Similarly, one finds that, for $k \ne n$ (and nonzero), $H_{k} ( 
D^{n} / S^{n-1} ) =0$.

\paragraph{Maps and homomorphisms}

Given a continuous map $f : X \to Y$, there is an induced map in 
homology, which we shall denote by $H_{*} (f)$ or $f_{*}$
\[ H_{*} (f) : H_{*} (X) \to H_{*} (Y) \]
and one checks that homology is a {\bf covariant functor} from the category 
{\bf Top} $= ( \text{ 
Top}, C^{0} )$ of topological spaces and continuous maps to the 
category {\bf Ab}$= ( \text{Ab}, \text{Hom} )$ of abelian groups and 
homomorphisms between them.
Since $H_{*} (X)$ is graded, the above homomorphism is understood to 
mean that it consists of homomorphisms at each level of homology:
\[ H_{k} (f) : H_{k} (X) \to H_{k} (Y) \]
for all $k$.

In fact, it would be more precise to say that the functor goes from 
the category {\bf hTop} of homotopy equivalence classes of spaces and homotopic 
maps to the category {\bf Ab}, since
\begin{prop}
If the maps $f$ and $g$ are homotopic, then the maps in homology 
coincide: $f_{*} = g_{*}$.
\end{prop}
and
\begin{prop}
If two spaces $X$ and $Y$ are homotopy equivalent and the map $f$ has 
a homotopy inverse, then $f_{*}$ is an isomorphism and the homology 
groups of $X$ and $Y$, $H_{*} (X)$ and $H_{*} (Y)$ are isomorphic.
\end{prop}

\begin{cor}
If $f$ is a homeomorphism of spaces, then $f_{*}$ is an isomorphism.
\end{cor}

When the space $X$ is a finite-dimensional manifold, the homology 
groups $H_{k} (X; \Z )$ are finitely generated; thus, in this case, 
the basic structure theorem for finitely generated abelian 
groups is applicable.

\begin{theorem}[Structure Theorem]
Any finitely generated abelian group $G$ decomposes uniquely as the 
direct sum
\[ G = F \oplus \tau \]
where the abelian group $F$ is free and the group $\tau$ is a torsion 
subgroup.
\end{theorem}
In fact, one can describe the torsion group $\tau$ is more detail 
(see, for example,~\cite{lang:alg}.)

The dimension of the free part of the homology group $H_{k}$ is called 
the $k$th-{\bf Betti number}, $b_{k} = \dim H_{k} (X ; \Z )$.
The {\bf Euler characteristic} $\chi (X)$ is the alternating sum of 
the Betti numbers
\[ \chi (X) = \sum _{k} {(-1)}^{k} b_{k} . \]

\section{Collections of topological necessary conditions}

The definition of certain Gauss maps is helpful in the statement of 
our results.
We shall assume that $M^{n} = \Rn$ or is an open subset of it.

\begin{definition}
\begin{enumerate}

\item Suppose the vector field $X$ is nowhere zero in $M^{n}$.
Then the {\bf Gauss map} $G_{X} : M^{n} \to S^{n-1}$ is defined by 
\[ x \mapsto \frac{X(x)}{ | X(x) | } .\]

\item Suppose that $N^{n-1} \subset M^{n}$ is a submanifold such that 
the restriction of the vector field $X$ to $N$ is nowhere zero.
Then the {\bf Gauss map} $G_{X | N} : N^{n-1} \to S^{n-1}$ is 
obtained by restricting the Gauss map $G_{X}$ to $N$.
Note that this is a map 
between two manifolds of the same dimension, one of which is a sphere.

\item If the submanifold $N^{n-1} \subset M^{n}$ is orientable, we 
define the {\bf Gauss map} $G_{N} : N^{n-1} \to S^{n-1}$ by  mapping 
$x \in N$ to the unit normal vector to $N$ at $x$ (where an `outward' 
direction is fixed by choosing an oriented basis  on $N$ and 
completing it to a basis of $\Rn$ consistent with an orientation of 
$\Rn$.) 
Note again that the Gauss map is a map from an $(n-1)$ dimensional 
space to the $(n-1)$-sphere.
\end{enumerate}
\end{definition}

\subsection{Index-Theoretic Necessary Conditions}

The \emph{topological index} of equilibrium points leads to a 
number of necessary conditions for achieving dynamics with equilibrium 
points of given stability.
These are {\emph{global} results and are rather classical; our only novelty 
is in trying to use as modern an algebraic topological framework as we 
can to express them.

If $e \in M^{n}$ is an isolated equilibrium point of the vector field 
$X$, take a ball neighborhood $U$ of $e$ (an open set homeomorphic to 
a ball) such that $e$ is the \emph{only} equilibrium of $X$ in $U$ 
and its boundary $N = \partial U$ is a closed submanifold 
homeomorphic to a sphere.
Then the Gauss map $G_{X|N}$ gives a map from the sphere $S^{n-1}$ to 
itself
\[ S^{n-1} \stackrel{h}{\rightarrow} N \stackrel{G_{X|N}}{\rightarrow} 
S^{n-1} \]
where $h^{-1}$ is the homeomorphism from $N$ to the sphere.

At the level of homology, we thus get a homomorphism $\psi = G_{X|N} 
\circ h$ from $H_{n-1} ( S^{n-1} )$ to itself.
Since this group is isomorphic to $\Z$, we get a homomorphism from 
$\Z to \Z$.
Since $\Z$ is a principal ideal domain, such maps are specified by 
the image of the generator, say $\alpha \in H_{n-1} ( S^{n-1} )$.
If, say, $\psi ( \alpha ) = k \alpha $, then $k$ is the {\bf 
topological index} of the equilibrium $e$.\footnote{Confusingly, we 
are about to give a theorem where the term `degree' is used instead 
of `toplogical index'; the two terms are equivalent.
We shall try use the qualifier `topological' to avoid confusion with 
other uses of the term index.}
It does not depend on the precise $U$ chosen.

The classical theorem of Hopf describes maps from the sphere to itself.

\begin{theorem}[Hopf's Classification Theorem]
Homotopy equivalence classes of maps from $S^{n-1}$ to itself are in 
a one-to-one correspondence with the integers.
For each integer $k$, the class of maps corresponding to it is called 
the class of maps of {\bf degree} $k$.
\end{theorem}

For a hyperbolic equilibrium point of \emph{stability index} $k$, 
the \emph{topological index} (or degree) is equal to ${(-1)}^{n-k}$.
Degree $k$ maps are easily obtained from the degenerate equilibria at 
the origin of the system in complex form: $\dot z = z^{k}$, for $k 
\ne 0$.

The Hopf classification of maps from the sphere to itself has a 
crucial generalization to maps of an arbitrary compact manifold of 
dimension $n-1$ to a sphere of dimension $n-1$ (see 
Whitehead,~\cite{whiteh}, p.244)
\begin{theorem}[Hopf-Whitney]
The homotopy equivalence classes of maps of an $(n-1)$-dimensional 
compact manifold $N^{n-1}$ to the sphere $S^{n-1}$ are in one-to-one 
correspondence with the elements of the cohomology group $H^{n-1} ( 
N^{n-1} ; \Z )$.
\end{theorem}
\begin{cor}
If $N$ is orientable, then the homotopy equivalence classes of maps 
from $N^{n-1}$ to $S^{n-1}$ are in one-to-one correspondence with the 
integers; they are thus again classified by `degree.' 
\end{cor}
This is, of course, because, for any orientable manifold, $H^{n-1} ( 
N^{n-1} ; \Z ) \simeq \Z$.
If $N$ is {\bf not orientable}, then this group is $\Z _{2}$ and two 
maps are homotopic iff they have the same mod-$2$ degree.

The {\bf global} version of the Hopf index classification result is 
the following theorem of Poincar\'{e}-Hopf
\begin{theorem}[Poincar\'{e}-Hopf]\label{PH}
\begin{enumerate}
        \item  Suppose $W^{n} \subset \Rn$ is a compact subset with nonempty 
        interior such that its boundary is an $(n-1)$-dimensional submanifold 
        of $\Rn$.
        Suppose $X$ is a vector field on $\Rn$ that is nowhere zero on the 
        boundary $\partial W$ and has a finite set of equilibrium points $E$.
        Then
        \begin{equation}
                \deg G_{X| \partial W} = \sum _{e_{i} \in E \cap W} \ind e_{i}
                \label{PH1}
        \end{equation}

        \item  Suppose $M^{n}$ is a compact manifold and $X$ is a vector 
        field on $M^{n}$ with a finite number of isolated equilibria.
        If the boundary of $M^{n}$ is not empty, we require the vector field 
        to point inwards at all points.
        The we have
        \begin{equation}
                \sum _{e_{i} \in E} \ind e_{i} = {(-1)}^{n} \chi ( M^{n} )
                \label{PH2}
        \end{equation}
    where $\chi ( M^{n} )$ is the Euler characteristic of the 
    manifold $M^{n}$ and $E$ is the set of equilibrium points.
    
    In particular, the sum of the topological indices of the 
    equilibria is a topological invariant of the manifold and thus is 
    independent of the vector field chosen.
        \item  Suppose $W^{k}$ is any submanifold of $\Rn$, with $0 \le k \le 
        n-1$.
        Consider a tubular neighborhood $N_{\ep} ( W^{k} )$ so that 
        $\partial N_{\ep} ( W^{k} )$ is an $(n-1)$-dimensional submanifold 
        of $\Rn$.
        If $X$ is any vector field on $\Rn$ such that, on $W^{k}$, $X$ has a 
        finite number of nondegenerate equilibria, then
        \begin{equation}
                \sum _{e_{i} \in E \cap W} = \deg G_{X| N_{\ep} ( W^{k} )}
                \label{PH3}
        \end{equation}
\end{enumerate}
\end{theorem}
We have collected different versions of this important theorem to help 
the reader find the most convenient form for extracting topological 
information in applications.
Milnor~\cite{milnor:tdv} proves versions (2) and (3) and contains a 
nice discussion.

\begin{remark}
The index already contains considerable topological information for 
the purposes of extracting necessary conditions.
For the case of an asymptotically {\bf attracting equilibrium}, for example, 
the topological index is equal to ${(-1)}^{n}$, which means that the generator 
of $H_{n-1} ( S^{n-1} )$ is mapped to itself or its negative, 
depending on the parity of $n$.
As a result, the Gauss map is an isomorphism in homology and we 
conclude that it must then be surjective and injective.
The surjectivity is essentially the {\bf Krasnosel'skii-Brockett} 
condition and the injectivity was derived by {\bf Coron}.
The form we have given is, however, considerably more general. 
\end{remark}
\begin{remark}
It must be emphasized that the index is \emph{`blind'} to all other 
dynamical features except equilibria.
Looking at the same point from the other side of the equalities in 
Theorem~\ref{PH}, the topological type of the Gauss map in the 
large (on the boundary of an enclosing set) affects the configuration 
of equilibria inside---and fixes the sum of their indices.
\end{remark}

A few examples as simple illustrations of the statements of the 
theorem:

\begin{example}
In $\Rn$, a ball with a vector field pointing inwards at the boundary 
must contain equilibria whose index sum is ${(-1)}^{n}$.
If these are all hyperbolic, then the options are
\begin{itemize}
        \item  A single attracting equilibrium.

        \item  Two attractors and a one-saddle.

        \item  If $n$ is even, a single repeller is not ruled out; notice 
        that the two cases can be distinguished using the Conley index, since 
        the exit set differs for the two cases.

        \item  Any other configuration of equilibria with the same net index 
        sum.
\end{itemize}
\end{example}

\begin{example}
In $\R ^{3}$, an embedded {\bf torus} $T^{2}$ gives possible Gauss 
maps of arbitrary degree, since its top homology is equal to $\Z$.
If, however, we know that there are no enclosed equilibria, as for 
example in the case where the torus isolates a limit cycle, then the 
degree must be zero, by part (1) of Theorem~\ref{PH}, independently 
of the stability type of the limit cycle.

This means that the Gauss map is homotopic to the constant map and 
hence \emph{does not have to be onto} (it is not onto in general, for a small 
enough torus around the limit cycle).
Thus, no necessary condition is derivable in this case, whether the 
limit cycle is stable or not.
\end{example}

\begin{example}
On the torus $T^{2}$ we have, by part (2) of the Theorem, that any 
vector field must have total index sum equal to zero, since the Euler 
characteristic of the torus is zero.
Thus, vector fields that everywhere nonzero are permissible 
topologically, as are vector fields with one attractor and one saddle, 
one repeller and a saddle, one attractor, one repeller and two saddles 
etc.
\end{example}

\subsection{Necessary conditions using the topological index}

It should be clear from the examples how to derive necessary 
conditions for achieving global dynamics from the index theorems.

Suppose given a Morse specification of gradient type, $\cM = ( E, 
h_{0} )$, with Morse-lyapunov functions $\cF ( \cM )$.
In the state space manifold, any choice of an oriented hypersurface 
that avoids $|E|$ has a Gauss map degree fixed by the sum of the 
indices of the enclosed `equilibria'.
If this is non-zero, this implies that there must exist control 
sections such that the controlled dynamics give a Gauss map with the 
desired property.
In particular, if the index sum is equal to plus or minus one, then 
the Gauss map is onto.
Let us remark that the conditions obtained can iether be used \emph{
locally} to check, for example, local stabilizability by requiring 
the map to have degree ${(-1)}^{n}$ for an arbitrarily small  
neighborhood of the equilibrium, or \emph{globally}, since the only 
relevant information is the position and stability of the desired 
equilibria and hence the index/degree results hold for any compact 
hypersurface avoiding $|E|$.

A more elegant algebraic topological way of checking \emph{simultaneously} all 
necessary conditions is the following (this does not make it easier 
to check in concrete cases):

% material from Nolta paper
We do first the case of local asymptotic stabilizability.

Suppose a control section $U \in \Gamma (D)$ is found that locally 
stabilizes the origin $0$ in some neighborhood $B$; it will be 
helpful to consider the set, for $\epsilon >0$,
\[
\Sigma _B = \{ (x,v) \in D {|}_{B} \; ; \; X(x) + v = 0 \} 
\] 
and the sequence
\begin{eqnarray*}
B \setminus \{ 0 \} \stackrel{\graph U}{\longrightarrow} B
\times \Rm \setminus \Sigma_B \stackrel{\iota}{\hookrightarrow} T \Rn
{|}_B \setminus \{ 0 \}  \\
\stackrel{G}{\longrightarrow} S \Rn {|}_B
\stackrel{\pi}{\longrightarrow} S^{n-1}  
\end{eqnarray*}
where $\graph U (x) = (x,U(x))$, $\iota$ is the inclusion map, $G$
is the Gauss map and $\pi$ is the obvious projection in the trivial
local sphere bundle. 

Since $0$ is an isolated equilibrium of $X+U$, $X+U \ne 0$ in $B
\setminus \{ 0 \}$ and the above is well-defined.

\subsection{A reinterpretation of Coron's condition}

With the tools we have at our disposal, it is now easy to give a more
geometric interpretation of the necessary condition for local feedback
stabilization given in~\cite{coron}:
We start by noticing that, if $B$ is a ball neighborhood of the 
equilibrium $0$, $B \setminus \{ 0 \}$ is homotopically
equivalent to $S^{n-1}$ (it actually retracts to the sphere).
Thus the composed map defined by the above sequence, call it $\phi$,
\[
\phi : B \setminus \{ 0 \} \to S^{n-1}
\]
has a well-defined degree, since $0$ is asymptotically stable for $X+U$
and this degree is equal to ${(-1)}^n$.
This means that, at the level of, for example, homology (or homotopy),
the generator, call it $\alpha$, of $H_{n-1} ( S^{n-1} ) \simeq \Z$ is
in the image of $\phi$.
In other words, if $0$ is LAS, then there is some local section such
that the degree of the above map is defined and the image of the
corresponding homomorphism at the level of homology is the whole of
$H_{n-1} ( S^{n-1} )$.
This is essentially Coron's result:
Consider the commutative diagram
\begin{equation}
\begin{array}{ccc}
        B \setminus \{ 0 \} & \rightarrow  & S^{n-1}  \\
        \downarrow & \nearrow &   \\
        D_ \setminus \Sigma _V  & &
\end{array}
\end{equation}
where the vertical map is inclusion and the map from $D_B \setminus
\Sigma _B$ to $S^{n-1}$ will be denoted also by $X+U$ and is given by
the composition $(x,v) \mapsto X(x)+v \mapsto G( X(x) +v) $.
We have that
\[
\phi _* ( H_{n-1} ( B \setminus \{ 0 \} )) = H_{n-1} ( S^{n-1} ) .
\]

\begin{theorem}[Coron, 1990]
If the system $(X,D)$ is locally asymptotically stabilizable, then
\[
{(X+U)}_* ( H_{n-1} ( D_B \setminus \Sigma _B ) ) = H_{n-1} ( S^{n-1}
). 
\]
\end{theorem}

\subsection{Generalizations}

The simple reasoning that led to Coron's result can be generalized to
equilibrium points that are not attractors, but have a well-defined
stability index.

\begin{theorem}
Let $0$ be an equilibrium of the state dynamics $X$ of the control pair
$(X,D)$.
If there is a continuous local feedback that yields dynamics $X+U$ with
$0$ an equilibrium of index $k \, , \, 0 \le k \le n $, then
\[
{( X+U })_* ( H_{n-1} ( D_B \setminus \Sigma _B )) = H_{n-1} ( S^{n-1}
) .
\]  
\end{theorem}

Finally, necessary conditions applicable to an arbitrary
compact, connected IIS $\cS$, isolated by the set $B \subset M^n$ can 
be given.
More explicitly, we assume that there is a local feedback $U: B \to
D$ such that $X+U$ has an IIS $\cS$, whose dynamical structure is known
(for example, $\cS$ as a set consists of a number of equilibria and
limit cycles and their connecting orbits.)
Notice that $X+U \ne 0$ in $B \setminus \cS$.

\begin{lemma}
Endow $M^n$ with a Riemannian metric.
There is a function $h$ defined on $B \setminus \cS$ such that its
gradient vector field $\nabla h$ is topologically equivalent to $X+U$ and
such that the Gauss maps of $\nabla h$ and $X_U$ induce the same
homomorphisms on homology, both $G_{\nabla h}$ and $G_{X_U}$ mapping
\[
H_* ( V \setminus \cS ) \rightarrow H_* ( S M^n {|}_{V \setminus \cS} ).
\]
\end{lemma}
Now, as we did for the case of local stabilization, we have the map
$\phi$ defined by the composite map below 
\begin{equation}
D_{B \setminus \cS} \setminus \Sigma  \stackrel{X+v}{\rightarrow} 
T M^n \setminus \Sigma \stackrel{G}{\rightarrow}  S M^n 
% & \searrow & \downarrow & \swarrow \\
% && V \setminus \cS   
\end{equation} 
which induces the map $\phi _*$ in homology
\[
H_* ( D_{B \setminus \cS} ) \rightarrow H_* ( S M^n {|}_{B \setminus
\cS} ) .
\]
We now have the result
\begin{theorem}
If the control pair $(X,D)$ can achieve dynamics with IIS $\cS$
isolated by the set $B$, then the images of the maps $\phi _*$ and
$G_{- \nabla h}$ in $ H_* ( S M^n {|}_{B \setminus \cS} )$ coincide.
\end{theorem}

\section{Homotopy equivalence and homotopic results}

An elementary, but fundamental result forms the key to an alternative 
approach to the derivation of necessary conditions.
It concerns the Gauss maps of a gradient vector field of a Lyapunov 
function for the dynamics $X$ and the Gauss map of the dynamics on 
level sets of the Lyapunov function.

\begin{theorem}
Let $\cV ^{n-1}$ be  a compact regular level set of some Lyapunov 
function $V$ for the dynamics $X$ on $M^{n} \subset \Rn$.
Then, the Gauss maps $G_{X| \cV}$ and $G_{- \nabla V | \cV}$ os the 
vector field $X$ and of the gradient vector field of $V$ with respect 
to any riemannian metric are homotopy equivalent.
\end{theorem}

\begin{proof}
Decompose the tangent space $T M^n {|}_{\cV}$ into the tangent space of
$\cV$ and the span of the gradient vector field $\nabla V$.
If $X_n$ is the projection of $X$ to the span of $\nabla V$,we have that
$X_n$ is nowhere zero on $\cV$.

Consider the isotopy of vector field
\[ Y_t (x) = (1-t) X_n (x) + t X(x) , \; 0 \le t \le 1 . \]
We have that $Y_0 = X_n$ and $Y_1 = X$.

Now notice that this gives an isotopy for the corresponding Gauss maps as
well: this is because $Y_t (x) \ne 0$ on $\cV$ and for all $t$.
To see this, write $Y_t$ as
\[ Y_t = X_n + t ( X - X_n ) \]
and notice that the vector field $X- X_n$ is orthogonal to $X_n$, which is
everywhere nonzero.

Define the Gauss maps parametrized by $t$
\[ G_t : \cV^{n-1} \to S^{n-1} , \; x \mapsto \frac{Y_t (x)}{| Y_t (x) |}.\]
Since $Y_t (x)$ is everywhere nonzero, this is well defined and gives an
isotopy between
\[ G_0 = \frac{X_n}{| X_n |} = \frac{- \nabla V}{| - \nabla V |} = G_{-
  \nabla V} \]
and
\[ G_1 = \frac{X}{| X |} =  G_{X} . \]
\end{proof}

For reference purposes, let us denote the set of homotopy equivalence
classes of maps between two spaces $\Omega$ and $\Omega '$ by
\[ [ \Omega , \Omega ' ] \]
according to the standard notation.
Given a map $f: \Omega \to \Omega '$, we write $[f]$ for its equivalence
classs.
We thus have, in this notation, that
\[ [ G_X ] = [ G_{- \nabla V} ] , \text{in} \; [ \cV^{n-1} , S^{n-1} ] . \]

\paragraph{Relations to the index}

Since the spaces involved are of the same dimension and the target space is
a sphere, we have, by the Hopf theory, that these homotopy equivalence
classes are classified by degree.

\paragraph{Limit Cycles}

In the case of a limit cycle $\gamma$, we saw that the Gauss map always has
degree zero.
Additional necessary conditions are obtained by examining the Gauss map in
more detail.

\begin{theorem}\label{limc}
Suppose $\gamma$ is a limit cycle for the dynamics $X$ on $\Rn$.
then
\begin{enumerate}
\item For any $\ep >0$, there is a neighborhood $N_{\delta} ( \gamma )$ such
  that
\[ G_X ( N_{\delta} ( \gamma ) ) \subset N_{\ep} ( G_X ( \gamma )) . \]
\item The image $G_X ( \gamma )$ is not contained in any hemisphere: in
  other words, for any hyperplane $\cP \subset \Rn$, $G_X ( \gamma ) \cap
  \cP \ne \emptyset$.
Moreover, for generic $\cP$, $| G_X ( \gamma ) \cap \cP |$ is even (here the
bars denote cadinality of a finite set.)
\end{enumerate}
\end{theorem}

\begin{proof}
The first part is proved by continuity and the long flow box
(see~\cite{pdem}.)

The second part is by contradiction: suppose there exists a hyperplane $\cP
_a = \{ v \in \Rn \; ; \; a(v) = 0 \}$, for some $a \in {( \Rn )}^*$ and is
such that $G_X ( \gamma ) \cap \cP _a = \emptyset$.
Since any hyperplane separates $S^{n-1}$ into two parts, we must have that
$a ( G_X ( x ) )$ is of uniform sign, say negative, for all $x \in \gamma$.

Choose a basis $b_1 , \ldots , b_n$ of $\Rn$ such that $a$ is the dual basis
vector of $b_1$, i.e. $a( b_1 ) = 1$ and $a ( b_i ) = 0$ for all $i \ne 1$.
Write $x_1 , \ldots , x_n$ for the coordinates in this basis.

\begin{claim}
The function $V(x) = \frac{1}{2} x_1^2$ is a Lyapunov function for $X$ in
some open neighhborhood of $\gamma$.
\end{claim}
This is shown by computing $\frac{dV}{dt}{|}_{\gamma}$.
We have
\[ \frac{dV}{dt} = ( x_1 , 0 , \ldots , 0 ) \cdot \dot \gamma \]
and, since $G_X = \frac{X}{|X|}$, this is just $a(X) < 0$.

The claim now establishes a contradiction that proves the theorem, since
$G_X ( \gamma )$ is a closed curve.
The last part also follows from this fact and an elementary transversality
argument.
\end{proof}

Theorem~\ref{limc} says roughly that, even though the image of the Gauss map
of a limt cycle is `thin,' still it must curve sufficiently in the target
sphere so as to intersect all possible hyperplanes.

As for the Lyapunov level sets near a limit cycle, we have

\begin{theorem}
Suppose $\gamma$ is a stable limit cycle for some controlled dynamics.
then, on each level set of a Lyapunov function near $\gamma$, each direction
(i.e. element of the unit sphere) appears at least twice, in other words,
for each $v \in S^{n-1}$,
\[ | G_{- \nabla V}^{-1} (v) | \ge 2 . \]
\end{theorem}
(The proof is a basic topolgical facts about tori and is omitted.)
Thus, even though the Gauss map of the gradient vector field of Lyapunov
functions is of degree zero on any level (as it should be by the
Poincar\'{e}-Hopf theorem~\ref{PH}), it covers the unit sphere at least
twice.

We see, therefore, that members of the \emph{same} homotopy equivalence
class of maps can have widely different Gauss images.
The trick, as far as control is concerned, is to find a representative
arising from a control section (see~\cite{contbk}.)

\section{Summary}

We have presented ways of deriving collections of necessary conditions for
achieving dynamics of a given type and we also pointed out the limitations
of such topological conditions (due to the simplicity of the Hopf theory of
maps to a sphere.)
The basic aim of any analysis is, of course, to arrive at \emph{constructive} methodologies.
In the treatment of this subject in~\cite{contbk}, we find that
conditions that are both necessary and sufficient can be found for achieving
dynamics in a certain class.
In this light, the fundamental source of necessary conditions is the class
of control-transverse sections and the resulting feedback-invariant dynamics.

\bibliography{bibcontrol}

\begin{thebibliography}{1}

\bibitem{coron}
J.~M. Coron.
\newblock A necessary condition for feedback stabilization.
\newblock {\em System and Control Letters}, 14:227--232, 1990.

\bibitem{gr:harp}
M.~Greenberg and J.~Harper.
\newblock {\em Algebraic topology: a first course}.
\newblock Benjamin/Cummins, Reading, Mass., 1981.

\bibitem{contbk}
E.~Kappos.
\newblock {\em Global Controlled Dynamics}.
\newblock Draft Manuscript, 2007.

\bibitem{lang:alg}
S.~Lang.
\newblock {\em Algebra}.
\newblock Addison-Wesley, Reading, Mass., 1971.

\bibitem{milnor:tdv}
John~W. Milnor.
\newblock {\em Topology from a Differentiable Viewpoint}.
\newblock The University Press of Virginia, Charlottsville, 1965.

\bibitem{pdem}
J.~Palis and W.~de~Melo.
\newblock {\em Geometric theory of dynamical systems: an introduction}.
\newblock Springer-Verlag, 1982.

\bibitem{whiteh}
George~W. Whitehead.
\newblock {\em Elements of homotopy theory}, volume~61 of {\em Graduate Texts
  in Mathematics}.
\newblock Springer-Verlag, 1978.

\end{thebibliography}
\bibliographystyle{plain}

\end{document}